\documentclass[12pt]{amsart}
\title[Poincar\'e polynomial for FC elements in type $A$ ]{Poincar\'e polynomial for fully commutative elements in the symmetric group}
\author{Sadek AL HARBAT and Corinne BLONDEL}
\address{} 
\email{sadekharbat@inst-mat.utalca.cl;  cblondel@math.univ-paris-diderot.fr}
\date{\today}

\thanks{Sadek Al Harbat was supported by Fondecyt Postdoctoral grant 3170544.}

\usepackage{etex} 
\usepackage[latin1]{inputenc}	
\usepackage[T1]{fontenc}
\usepackage[english]{babel}
\usepackage{amsmath,amssymb}
\usepackage{wasysym}
\usepackage{stmaryrd}
\usepackage[table]{xcolor}
\usepackage{color}
\usepackage{dsfont}\let\mathbb\mathds
\usepackage[all]{xy}
\usepackage{graphicx}
\usepackage{mathenv}
\usepackage{lastpage}
\usepackage{fancyhdr}	
\usepackage{subfigure}
\usepackage{geometry}
\usepackage[version=3]{mhchem}
\usepackage[force,almostfull]{textcomp}
\usepackage{array}
\usepackage{lipsum}		
\usepackage{eso-pic}
\usepackage{multirow}
\usepackage{fancybox}
\usepackage{cite}
\usepackage{amsthm}
\usepackage{listings}
\usepackage{lscape}
\usepackage{multicol}		
\usepackage{setspace}
\usepackage{times}
\usepackage{eurosym}
\usepackage{latexsym}
\usepackage{ulem}
\usepackage{lmodern}
\usepackage{colortbl}
\usepackage{rotating}
\usepackage{type1cm}
\usepackage{lettrine}
\usepackage{ragged2e}
\usepackage{mathrsfs}
\usepackage{scrtime}
\usepackage{enumitem}

\usepackage{longtable}
\usepackage{url}
\usepackage{nomentbl}
\usepackage{nomencl}
\usepackage{hyperref}

\usepackage{epigraph}

\hypersetup{
pdfpagemode=UseOutlines,      
pdfstartview=Fit,             
pdffitwindow=true,            
pdfpagelayout=TwoColumnsRight,
pdftoolbar=true,              
pdfmenubar=true,              
bookmarksopen=false,          
bookmarksnumbered=true,       
colorlinks=true,              
pdfauthor={Ton nom},          
pdftitle={Titre PDF},         
pdfcreator=PDFLaTeX,          %
pdfproducer=PDFLaTeX,         %
linkcolor=blue,               
urlcolor=blue,                
anchorcolor=black,            
citecolor=blue,               
frenchlinks=true,             
pdfborder={0 0 0}             
}
\usepackage{pgf, tikz}
\usetikzlibrary{matrix,positioning}
\usetikzlibrary{arrows,decorations.markings}

\newtheorem{theorem}{Theorem}[section]

\newtheorem{definition}[theorem]{Definition}
\newtheorem{proposition}[theorem]{Proposition}
\newtheorem{lemma}[theorem]{Lemma}
\newtheorem{corollary}[theorem]{Corollary}

\newtheorem{remark}[theorem]{Remark}

\addto\captionsfrenchb{}
\addto\captionsfrenchb{}

    \newlength{\myarrowsize} 
    \newlength{\myoldlinewidth}

    \pgfarrowsdeclare{myto}{myto}{
        \pgfsetdash{}{0pt} 
        \pgfsetbeveljoin 
        \pgfsetroundcap 
        \setlength{\myarrowsize}{0.5pt}
        \addtolength{\myarrowsize}{.5\pgflinewidth}
        \pgfarrowsleftextend{-4\myarrowsize-.5\pgflinewidth} 
        \pgfarrowsrightextend{.7\pgflinewidth}
    }{
        \setlength{\myarrowsize}{0.4pt} 
        \addtolength{\myarrowsize}{.3\pgflinewidth}  
        \setlength{\myoldlinewidth}{\pgflinewidth}
        \pgfsetroundjoin
        \pgfsetlinewidth{0.0001pt}
        \pgfpathmoveto{\pgfpoint{0.43\myarrowsize}{0}}
        \pgfpatharc{0}{70}{0.14\myarrowsize}
        \pgfpatharc{-80}{-169.5}{4\myarrowsize}
        \pgfpatharc{150}{189}{0.95\myarrowsize and 0.95\myarrowsize}
        \pgfpatharc{0}{-40}{15\myarrowsize}
        \pgfpathmoveto{\pgfpoint{0.43\myarrowsize}{0}}
        \pgfpatharc{0}{-70}{0.14\myarrowsize}
        \pgfpatharc{80}{169.5}{4\myarrowsize}
        \pgfpatharc{-150}{-189}{0.95\myarrowsize and 0.95\myarrowsize}
        \pgfpatharc{0}{40}{15\myarrowsize}
        \pgfpathclose
        \pgfsetstrokeopacity{0.25}
        \pgfusepathqfillstroke
    }

    \pgfarrowsdeclare{myonto}{myonto}{
        \pgfsetdash{}{0pt} 
        \pgfsetbeveljoin 
        \pgfsetroundcap 
        \setlength{\myarrowsize}{0.5pt}
        \addtolength{\myarrowsize}{.5\pgflinewidth}
        \pgfarrowsleftextend{-4\myarrowsize-.5\pgflinewidth} 
        \pgfarrowsrightextend{.7\pgflinewidth}
    }{
        \setlength{\myarrowsize}{0.4pt} 
        \addtolength{\myarrowsize}{.3\pgflinewidth}  
        \setlength{\myoldlinewidth}{\pgflinewidth}
        \pgfsetroundjoin
        \pgfsetlinewidth{0.0001pt}
        \pgfpathmoveto{\pgfpoint{0.43\myarrowsize}{0}}
        \pgfpatharc{0}{70}{0.14\myarrowsize}
        \pgfpatharc{-80}{-169.5}{4\myarrowsize}
        \pgfpatharc{150}{189}{0.95\myarrowsize and 0.95\myarrowsize}
        \pgfpatharc{0}{-40}{15\myarrowsize}
        \pgfpathmoveto{\pgfpoint{-7\myarrowsize}{0}}
				\pgfpatharc{0}{70}{0.14\myarrowsize}
        \pgfpatharc{-80}{-169.5}{4\myarrowsize}
        \pgfpatharc{150}{189}{0.95\myarrowsize and 0.95\myarrowsize}
        \pgfpatharc{0}{-40}{15\myarrowsize}
        \pgfpathmoveto{\pgfpoint{0.43\myarrowsize}{0}}
        \pgfpatharc{0}{-70}{0.14\myarrowsize}
        \pgfpatharc{80}{169.5}{4\myarrowsize}
        \pgfpatharc{-150}{-189}{0.95\myarrowsize and 0.95\myarrowsize}
        \pgfpatharc{0}{40}{15\myarrowsize}
			  \pgfpathmoveto{\pgfpoint{-7\myarrowsize}{0}}
        \pgfpatharc{0}{-70}{0.14\myarrowsize}
        \pgfpatharc{80}{169.5}{4\myarrowsize}
        \pgfpatharc{-150}{-189}{0.95\myarrowsize and 0.95\myarrowsize}
        \pgfpatharc{0}{40}{15\myarrowsize}
        \pgfpathclose
        \pgfsetstrokeopacity{0.25}
        \pgfusepathqfillstroke
        \pgfusepathqfillstroke
    }

 \pgfarrowsdeclare{myhook}{myhook}{
        \setlength{\myarrowsize}{0.6pt}
        \addtolength{\myarrowsize}{.5\pgflinewidth}
        \pgfarrowsleftextend{-4\myarrowsize-.5\pgflinewidth} 
        \pgfarrowsrightextend{.7\pgflinewidth}
    }{
        \setlength{\myarrowsize}{0.6pt} 
        \addtolength{\myarrowsize}{.5\pgflinewidth}  
        \pgfsetdash{}{+0pt}
        \pgfsetroundcap
        \pgfpathmoveto{\pgfqpoint{-2pt}{-6\pgflinewidth}}
        \pgfpathcurveto
            {\pgfqpoint{4\pgflinewidth}{-4.667\pgflinewidth}}
            {\pgfqpoint{4\pgflinewidth}{0pt}}
            {\pgfpointorigin}
        \pgfusepathqstroke
    }

		\pgfarrowsdeclare{my to}{my to}
{
  \pgfarrowsleftextend{-2\pgflinewidth}
  \pgfarrowsrightextend{\pgflinewidth}
}
{
  \pgfsetlinewidth{0.8\pgflinewidth}
  \pgfsetdash{}{0pt}
  \pgfsetroundcap
  \pgfsetroundjoin
  \pgfpathmoveto{\pgfpoint{-5.5\pgflinewidth}{7.5\pgflinewidth}}
  \pgfpathcurveto
  {\pgfpoint{-4.0\pgflinewidth}{0.1\pgflinewidth}}
  {\pgfpoint{0pt}{0.25\pgflinewidth}}
  {\pgfpoint{0.75\pgflinewidth}{0pt}}
  \pgfpathcurveto
  {\pgfpoint{0pt}{-0.25\pgflinewidth}}
  {\pgfpoint{-4.0\pgflinewidth}{-0.1\pgflinewidth}}
  {\pgfpoint{-5.5\pgflinewidth}{-7.5\pgflinewidth}}
  \pgfusepathqstroke
}

		\pgfarrowsdeclare{mytodashed}{mytodashed}
{
 \pgfarrowsleftextend{-2\pgflinewidth}
 \pgfarrowsrightextend{\pgflinewidth}
}
{
 \pgfsetlinewidth{0.8\pgflinewidth}
  \pgfsetdash{{0.8cm}{0.8cm}}{0cm}
\pgfsetroundcap
\pgfsetroundjoin
  \pgfpathmoveto{\pgfpoint{-5.5\pgflinewidth}{7.5\pgflinewidth}}
\pgfpathcurveto
{\pgfpoint{-4.0\pgflinewidth}{0.1\pgflinewidth}}
 {\pgfpoint{0pt}{0.25\pgflinewidth}}
{\pgfpoint{0.75\pgflinewidth}{0pt}}
\pgfpathcurveto
{\pgfpoint{0pt}{-0.25\pgflinewidth}}
  {\pgfpoint{-4.0\pgflinewidth}{-0.1\pgflinewidth}}
 {\pgfpoint{-5.5\pgflinewidth}{-7.5\pgflinewidth}}
\pgfusepathqstroke
}

\tikzstyle{vecArrow} = [thick, decoration={markings,mark=at position
   1 with {\arrow[semithick]{open triangle 60}}},
   double distance=1.4pt, shorten >= 5.5pt,
   preaction = {decorate},
   postaction = {draw,line width=1.4pt, white,shorten >= 4.5pt}]
\tikzstyle{innerWhite} = [semithick, white,line width=1.4pt, shorten >= 4.5pt]

\newcommand\dyckpath[3]{

	\draw[help lines] (#1) grid +(#2,#2);
	\draw[dashed] (#1) -- +(#2,#2);
	\coordinate (prev) at (#1);
	\foreach \dir in {#3}{
		\ifnum\dir=0
		\coordinate (dep) at (1,0);
		\else
		\coordinate (dep) at (0,1);
		\fi
		\draw[line width=2pt,] (prev) -- ++(dep) coordinate (prev);
	};
}

	\makeatletter
	\newcommand\POSITION[3]{%
	\begingroup
	\@tempdim@x=0cm
	\@tempdim@y=\paperheight
	\advance\@tempdim@x#1
	\advance\@tempdim@y-#2
	\put(\LenToUnit{\@tempdim@x},\LenToUnit{\@tempdim@y}){#3}%
	\endgroup
	}
\makeatother

\addto\captionsenglish{}

\begin{document}

\maketitle

\begin{epigraphs}
\qitem{{\it S.Rogers:} We need a plan of attack!}%
{}
\qitem{{\it T.Stark:} I have a plan: ... attack!!}%
{
 \textsc{Avengers 2012, M.C.U.}}
\end{epigraphs}


	\begin{abstract}
	Let $W^c(A_n)$ be the set of  fully commutative elements of the Coxeter group $W(A_n)$. Let 
$$
a_n(q)= \sum_{w \in W^c(A_n)} q^{l(w)}   . 
$$
We compute  $a_n(q)$. 
	\end{abstract}

\tableofcontents
	
\section{Introduction}

{\it Full-commutativity. }  
In a Coxeter system $(W,S)$ a fully  commutative element (say an FC element) is an element of which any reduced expression can be arrived to from any other by commutation relations. 
 
We focus here on the Coxeter system $(W(A_n), S)$, 
$S=\{\sigma_1, \cdots , \sigma_n\}$,  that can be viewed as   the symmetric group $\mathfrak S_{n+1}$ of permutations of 
$\{1, \cdots, n+1\}$, generated by the $n$   elementary transpositions 
$\sigma_i=(i, i+1)$ for $1\le i \le n$.   
FC elements in $W(A_n)$ (AKA {\it $321$-avoiding permutations} of the symmetric group), forming the subset denoted by $W^c(A_n)$,    were used in the famous work of  V. Jones \cite{Jones_1986}  before they were officially defined and studied, for example  in the works of Graham and Stembrige \cite{Graham,St96}. Nowadays the study of full-commutativity has taken its own place, forming a nice theory relating the Coxeter groups theory to many others: not starting by diagram algebras and not ending by algebraic combinatorics.   \\  

 {\it Poincar\'e polynomials. }
  For a given  Coxeter system $(W,S)$ and a subset 
$H$ of $W$ we define:

\begin{equation}\label{Poincare}
H(q)=  \sum_{0\leq i} \# \{ w \in H; l(w)=i \} \   q^{i} = \sum_{w \in H} q^{l(w)}.
\end{equation} 

Among the various conventions in the literature, we choose, following Bourbaki  \cite{Bourbaki_1981}, Humphreys \cite{Humphreys},  Bjorner and Brenti 
\cite{BB},  to call it the Poincar\'e series of $H$ (or Poincar\'e polynomial of $H$, if $H$ is finite) rather than another meaningful name: "the Coxeter length generating fonction" of $H$. 
Actually we focus on the series of polynomials $a_n=W^c(A_n)(q)$. Let $$
f(x,q)= \sum_{0\leq n}  a_n x^{n}.
$$
 In the literature many authors, seeking brevity, refer to $f(x,q)$ as the length generating fonction of $W^c(A_n)$, when it is really  the generating fonction of the Poincar\'e polynomial or the generating fonction of the length generating fonction, not to mention the fact that sometimes  even a three-variable series $f(x,y,q)$   is called a length generating fonction.    
So we choose to call $a_n$ the Poincar\'e polynomial of $W^c(A_n)$ for the sake of distinguishing it from  generating fonctions. \\

  A generating fonction of the Poincar\'e polynomial for  FC elements in type $A$ was studied in \cite{Barcucci2001}  and was extended to the  affine type  $\tilde A$ in \cite{HJ}, while another expression for  $\tilde A$ and formulas  for FC elements in other affine cases appear in   \cite{Biagioliandal}. In this work we do not use  generating fonctions, not even as a technical tool. In a forthcoming work we compute directly the Poincar\'e polynomial of  $W^c(\tilde A_n)$, in which $W^c(A_n)$ forms the subset of elements of  affine length  0, which is the most difficult case, while for affine lengths equal to and greater than 1, the computation is remarkably easier,  using the normal form  established in \cite{A2016}, where the notion of affine length is defined.  \\

More generally,     our method  of computation of the Poincar\'e series,  
starting from the normal forms  of FC elements in the four infinite families of affine Coxeter groups, 
namely $\tilde A$, $\tilde B$, $\tilde C$  and   $\tilde D$, established by the first author,    
reduces this computation to elements of affine length $0$, that is, to the Poincar\'e polynomial for 
FC elements in the three infinite families of  finite Coxeter groups $A$, $B$ and $D$. Among those three,  
the case $A_n$ is the generic case. From these facts comes the importance of the   Poincar\'e polynomial for FC elements in type $A$ which is the focus of this work.  \\   

{\it  Catalan level. }  
 The starting  point of this work is a partition of $W^c(A_n)$,  the cardinal of which  is 
the famous Catalan number 
$C_{n+1}= \frac{1}{n+2}\binom{2(n+1)}{n+1}$. In   a future paper we explain many other partitions to re-count $W^c(A_n)$ and to count many distinguished subsets  of $W^c(A_n)$, that give many interesting partitions of the Catalan number among which  Narayana numbers and the Catalan triangle, the latter coming  from the very partition that we use in this work (see Remark \ref{triangle});  we go down from the polynomials to the numbers related to the Catalan number by specializing $q$ to 1. In this work we compute explicitly  $\sum_{w \in W^c(A_n)} q^{l(w)} $ rather than  $\sum_{0\leq i} \#\{ w \in W^c(A_n); l(w)=i \} q^{i}$, in which specializing $q$ to $1$    gives a  new partition of the Catalan number with the "Coxeter" color all over it, so we have an explicit "pretty sophisticated" $a_n$, which is the next-to-last  step to answer the question: let $r$ be a positive integer, how many  elements do we have in $W^c(A_n)$ which are of length equal to $r$? That is: explaining the obvious equality of  definition (\ref{Poincare}), supposing that $H$ is $W^c(A_n)$.

 \bigskip
The paper is organized as follows. 

In section 2, we recall a  normal form for FC elements in $W(A_n)$, following Stembridge,   and 
we  partition the set $W^c(A_n)$  into $\{1\}$ and the  subsets  $A_n^j$, 
$1 \le  j \le n$, of FC elements having a normal form with rightmost element the $j$-th  generator.  
We write a recurrence relation for the  Poincar\'e polynomials $a_n^j=A_n^j(q)$ and obtain the quite intriguing fact that  $a_n^j$ is a linear combination of $a_{n-1}$, ..., $a_{n-j}$ over $\mathbb Z[q]$, 
with coefficients depending on $j$, not on $n$, namely 
(\ref{expressionofanj}):  
$$
   a_n^{ j}  = \sum_{k=1}^j  q^{j-k+1} B_{j}^k(q) a_{n-k}    \quad (1 \le j \le n)      . 
$$
where the family  of polynomials  $ B_{j}^k$ in $\mathbb Z[q]$ is uniquely determined (Proposition \ref{propanj}). This leads us to the main recurrence relation (\ref{recurrenceforan})   for the Poincar\'e polynomial: 
$$
  q^{n} a_{n-1}  =   q  + q^2 + \cdots + q^n  -  \sum_{k=2}^n    q^{n-k+1} B_{n}^k   a_{n-k}  \qquad (n\ge 1). 
$$
We observe that the value at $1$ of the polynomial $B_j^k$ is plus or minus a binomial coefficient 
and that,  by
 specializing $q$ to $1$,  we obtain a recurrence relation for Catalan numbers very similar to the one attributed by Stanley to Ming Antu in \cite[B1]{Stanley}. \\

In section 3 we proceed to the computation of the polynomials  $B_j^k$. It is in fact more convenient to compute the polynomials $b_j^k$ defined by  $b_j^k= B_j^{j-k+1}$. The family $b_j^k$ for fixed $j$ and variable $k$ can be viewed as a family of polynomials 
intermediate between $  b_{j}^{j-1} =  1-q^2 - \cdots - q^j
$  and    $ b_j^1=   (1-q^2)  \cdots   (1-q^j)$ and indeed we are led to define polynomials 
$ \Pi(a,b) = \prod_{a\le i \le b}(1-q^i)$ for $1\le a \le b$ and 
$\Sigma\Pi (a,b)[e_1, \cdots,  e_u]  $,  the sum of all possible products obtained from $ \Pi(a,b) $ 
by removing   $e_i$ consecutive terms $(1-q^j)\cdots (1-q^{j+e_i-1})$ 
and replacing them by $(-q^j)$.  Using these as basic bricks, we 
  obtain in Theorem \ref{polbkj} a formula for $b_j^k$.   
 We do not know if this family of polynomials has been used in other contexts. \\   

We start section 4 with a general expression for a sequence satisfying a recurrence relation of the same shape as (\ref{recurrenceforan}) (Proposition \ref{generalprop2}). 
As a direct application, we get an expression for the Poincar\'e polynomial $a_n$ 
 in Theorem \ref{maintheorem}. Our goal is achieved, yet   we find of interest to record a property of the sequence $a_n$ that we noticed on the way, this is Proposition \ref{Cnnandan} that says that, up to a shift, the sequence $(a_n)$ is the first column of the inverse matrix of $(b_j^k)$. This leads to a slightly different formula for $a_n$, to finish we write   both formulas extensively.   \\

We thank Mathieu Florence and Luc Lapointe for their challenging  comments.

\section{Main recurrence relation and its  general solution}

\subsection{Fully commutative elements of type $A$}\label{notations}

			Consider the $A$-type Coxeter group  $W(A_{n})$ ($n$ being a positive integer; we let   $W(A_{0})= 1$) that has  the following presentation by  generators 
$\{\sigma_1, \cdots, \sigma_n\}$ and relations:    \\
		
				\begin{itemize}[label=$\bullet$, font=\normalsize, font=\color{black}, leftmargin=2cm,parsep=0cm, itemsep=0.25cm,topsep=0cm]

\item $\sigma_i^2= 1$ for $1\leq i\leq n$;  

					\item braid relations: 
	\begin{itemize}   
	\item 
$\sigma_{i}\sigma_{i+1}\sigma_{i} = \sigma_{i+1}\sigma_{i}\sigma_{i+1}$ where $1\leq i\leq n-1$, 
\item   commutation relations:

$\sigma_{i} \sigma_{j} =\sigma_{j} \sigma_{i} $  where $1\leq i,j\leq n$ and $ \left| i-j\right| \geq 2$.	\\	
\end{itemize}
    
				\end{itemize}

Any element $w$ in   $W(A_{n})$  can be expressed as a product of generators, say 
$w=\sigma_{i_1} \sigma_{i_2} \cdots \sigma_{i_s}$ with $s\ge 0$ (for $s=0$ we have an empty product,  equal to  $1$).
 Among such expressions of $w$, the ones for which $s$ is minimal   are called {\it reduced} and the corresponding (minimal) value of $s$  is called the (Coxeter) length of $w$ and denoted by $l(w)$.

      \begin{definition}
			Elements $w$ in   $W(A_{n})$  for which one can pass from any reduced expression to any other one only by applying commutation relations are called {\rm fully commutative elements}. We denote  by $W^c(A_{n})$ the set of fully commutative elements in $W(A_{n})$. 
		    \end{definition}

Our aim in this paper is the computation of the length polynomial of fully commutative elements  in $W(A_{n})$, namely, denoting by $q$ the indeterminate:   
 $$
a_n = a_n(q) = \sum_{w \in W^c(A_{n})}  q^{l(w)}   \quad    (n \ge 0).
$$
	
 The basis of our enumeration of $W^c(A_n)$    is the following normal form of fully commutative elements, for which we refer to  Stembridge.  	 We let: 		
$$
\begin{aligned}
\lceil  i,j \rceil  &= \sigma_i \sigma_{i+1} \dots \sigma_j   \    \text{ for } 1\le i\le j \le n \     . 
  \end{aligned}
$$

\begin{theorem}\label{1_2}{\rm \cite[Corollary 5.8]{St}}
 Let $n$ be a positive integer, then   $W^c(A_n)$ is the set of elements of the   form: 
 \begin{equation}\label{eq:Stembridge}
 \lceil i_1, j_1 \rceil \lceil i_2, j_2 \rceil \dots  \lceil i_p, j_p \rceil  , 
 \text{ with } 0\le p \le n  \text{ and } 
 \left\{ \begin{matrix}
 n\ge j_1 > \dots > j_p \ge 1 ,  \cr    n\ge i_1 > \dots > i_p \ge 1,   \cr
 j_t \ge i_t   \text{ for }  1\le t \le p.  
 \end{matrix}\right.
 \end{equation}
\end{theorem} 

 It is well-known that the number of FC elements in $W^c(A_n)$ is the Catalan number 
$C_{n+1}= \frac{1}{n+2}\binom{2(n+1)}{n+1}$. We refer the reader to  \cite{St}
for comments on this fact, and 
to the list of 
214 appearances of the Catalan number established by Stanley in \cite{Stanley}: appearance number 107   indeed counts sequences of integers  satisfying (\ref{eq:Stembridge}).

\subsection{Groundwork: two recurrence relations} 

For $1\le j \le n$, 
we define $ A_n^{ j} $ as the set of fully commutative elements  in $W(A_n)$ given by a normal form
(\ref{eq:Stembridge}) that ends  with $\sigma_j$ on the right, in other terms such that $j_p=j$.  We are looking for the length polynomial of these elements:
$$
a_n^{ j} = a_n^j(q) = \sum_{w \in A_n^{ j}}  q^{l(w)}    \text{ for } 1 \le j \le n.
$$
We   make the convention that $ a_n^{ j} = 0$ for $j > n$ and  $ a_n^{0}=0$, for a reason that will appear soon. The length polynomial of  elements in $W^c(A_n)$ is 
$$
a_n =1+  \sum_{j=1}^n   a_n^{ j}  \quad (n\ge 1),   \quad   a_0=1 . 
$$

\begin{remark}\label{triangle}

We have already counted $a_n^{ j} (1)$ which is the number of FC elements ending with $\sigma_j$ on the right, that is:
$$
a_n^{ j} (1)= \frac{j}{n+1}  \binom{2n-j+1}{n}. 
$$

And this is but the famous Catalan's triangle, of which we give a full description in a forthcoming work. 
In other words, the relation above is a $q$-version of the Catalan triangle. 
\end{remark}

\medskip 

We partition  $ A_n^{ j} $  as follows. Let $w \in A_n^{ j}$ for $1 \le j \le n$, given by its normal form (\ref{eq:Stembridge}). Then either 
$\sigma_1$ does not appear in $w$, i.e. $i_p>1$, hence   $w $ corresponds, upon shifting the generators ($\sigma_i \mapsto \sigma_{i-1}$ for $i \ge 2$),  to a unique $w'\in A_{n-1}^{ j-1}$ with the same length; 
or $\sigma_1$ appears in $w$, i.e. $i_p=1$,  hence the rightmost  bloc in $w$ is 
$\sigma_1 \cdots \sigma_j$, of length $j$, and it is preceded on the left by either $1$ or an element that, upon the same shift as before,  belongs to $W^c(A_{n-1})$  with a normal form ending with some $\sigma_s$ with $s$ at least equal to $j$. We obtain:
\begin{equation}\label{basicrelation}
a_n^{ j} =  a_{n-1}^{ j-1} + q^{j} (1+ \sum_{s=j}^{n-1}  a_{n-1}^{s})   \quad (1\le j  \le n) . 
\end{equation}
We note that this relation does not hold for $j>n$ or $j=0$. For $j=1$   we get  
$$
  a_n^{1} = q a_{n-1}.   
$$
For $n \ge j\ge 2$ we can write as well 
	$$a_n^{j-1} = a_{n-1}^{j-2} + q^{j-1}(1+  \sum_{s=j-1}^{n-1}  a_{n-1}^{s}). $$
Substracting   $q$ times this last equation to  the previous one, we obtain  
\begin{equation}\label{basicrecurrence}
   a_n^{ j}  -q a_n^{j-1} = a_{n-1}^{ j-1}-  q   a_{n-1}^{j-2}          - q^j 
     a_{n-1}^{j-1}  \quad (2 \le j \le n)  .     
\end{equation}    
Using this relation for $j=2$ gives:
$$
a_n^2= q a_n^1 + (1-q^2) a_{n-1}^1 = q^2 a_{n-1} + q(1-q^2) a_{n-2}  \quad (n \ge 2) .  
$$

We claim the following: 

\begin{proposition}\label{propanj}
The initial conditions 
$$B_{j}^1(q)=1 \text{ for } j\ge 1,  \qquad 
B_{j}^j(q)=   (1-q^2) \cdots (1-q^j)   \text{ for } j\ge 2  $$  
and  the recurrence relation:  
\begin{equation}\label{recurrenceBkj}
 B_{j}^k(q)   =  B_{j-1}^k(q)+(1-q^j)  B_{j-1}^{k-1}(q)-
 B_{j-2}^{k-1}(q)    \quad (2 \le k \le j-1)      
\end{equation}
define  a unique family of polynomials $(B_{j}^k)_{1\le k \le j}$, in the variable $q$, with integer coefficients. Those  polynomials satisfy  $B_{j}^k(0)=1$. We have the following equality: 
\begin{equation}\label{expressionofanj}
   a_n^{ j}  = \sum_{k=1}^j  q^{j-k+1} B_{j}^k(q) a_{n-k}    \quad (1 \le j \le n)      . 
\end{equation}
\end{proposition}  
\begin{proof} 
Existence and unicity of the family  $(B_{j}^k)_{1\le k \le j}$ are an immediate consequence of the recurrence relation (\ref{recurrenceBkj}): thinking of $(j,k)$ on a grid with $j$ on the $x$ axis and $k$ on the $y$ axis, we see that the knowledge of $B_\ast^{k-1} $, corresponding to the line $y=k-1$,  and of $B_{k}^k$, the leftmost point on the line $y=k$,  
implies the knowledge of $B_\ast^k$, i.e. the line $y=k$. Since the initial conditions give us the bottom line $k=1$ and the diagonal $j=k$, we are done. The value at $0$ is easy. 

We now  prove  (\ref{expressionofanj})  by induction on $j$. The cases 
$j=1$ and  $j=2$ have been established above. We assume that (\ref{expressionofanj}) holds for any $t$ with $
1 \le t <j\le n$ and we use (\ref{basicrecurrence}) to prove it for $j$, $3 \le j \le n$. Indeed:  
$$
\begin{aligned}
  a_n^{ j}  &=  q a_n^{j-1} + (1-q^j) a_{n-1}^{ j-1}-  q   a_{n-1}^{j-2}    \\
&=  \sum_{k=1}^{j-1}  q^{j-k+1} B_{j-1}^k(q) a_{n-k}  + (1-q^j)
(\sum_{k=1}^{j-1}  q^{j-k} B_{j-1}^k(q) a_{n-1-k}  ) \\
&\qquad \qquad   -\sum_{k=1}^{j-2}  q^{j-k} B_{j-2}^k(q) a_{n-1-k}   \\
&=   \sum_{k=1}^{j-1}  q^{j-k+1} B_{j-1}^k(q) a_{n-k}  + (1-q^j)
(\sum_{k=2}^{j}  q^{j-k+1} B_{j-1}^{k-1}(q) a_{n-k}  ) \\
&\qquad \qquad   -\sum_{k=2}^{j-1}  q^{j-k+1} B_{j-2}^{k-1}(q) a_{n-k}   \\
&=  q^j  B_{j-1}^1(q) a_{n-1}  + \sum_{k=2}^{j-1}  q^{j-k+1}( B_{j-1}^k(q)\!+(1-q^j)  B_{j-1}^{k-1}(q) \!-
 B_{j-2}^{k-1}(q)   )  a_{n-k}  \\
&\qquad \qquad  +   q  (1-q^j)  B_{j-1}^{j-1}(q)  a_{n-j}   
\end{aligned}    
$$
so (\ref{expressionofanj}) holds for  $j$. 
\end{proof}

In what follows we shorten $B_j^k (q)$ to $B_j^k$. 

\subsection{Recurrence relation for the Poincar\'e polynomial}\label{subsection1dot2}
 Now (\ref{expressionofanj})  leads  to:   
$$\begin{aligned}
a_n&= 1 + \sum_{j=1}^n  \sum_{k=1}^j  q^{j-k+1} B_{j}^k a_{n-k}   
=  1 +    \sum_{k=1}^n  \left[\sum_{j=k}^n    q^{j-k+1} B_{j}^k   \right] a_{n-k} . 
\end{aligned}
$$
This is a recurrence relation that might allow   to compute $a_n$ from $a_{n-1}, \dots, a_0$. But we can do better. From the above we have, for $n\ge 1$:  
$$
  a_n^{ n}  = \sum_{k=1}^n    q^{n-k+1} B_{n}^k  a_{n-k}. 
$$
We can also compute directly $a_n^n$:  if $j_p=n$, then $p=1$ and 
$1\le    i_p \le n$ so that 
$$
a_n^n  =   q + q^2 + \cdots + q^n =  q \frac{1-q^n}{1-q}.  
$$
We get the recurrence relation:
\begin{equation}\label{recurrenceforan}
  q^{n} a_{n-1}  =   q  + q^2 + \cdots + q^n  -  \sum_{k=2}^n    q^{n-k+1} B_{n}^k   a_{n-k}  \qquad (n\ge 1). 
\end{equation} 

We postpone to the last section the study of this recurrence   relation. In the next section we compute  the polynomials $B_j^k$,  actually the 
related polynomials $b_j^k$.

\subsection{Value at $1$ and Catalan number} 
The family of values  $B_j^k(1)$ is uniquely defined by 
 Proposition  \ref{propanj}   at $q=1$,    with recurrence relation 
$$B_{j}^k(1)   =  B_{j-1}^k(1)-
 B_{j-2}^{k-1}(1)    \quad (2 \le k \le j-1)    .$$  
We check that  
$B_j^k(1)=   (-1)^{k-1}  \binom{j-k}{k-1}$   
and  write 
(\ref{recurrenceforan}) at $q=1$, replacing the various  $a_k(1)$ by Catalan numbers. We get:   
$$   C_n  =  n  -  \sum_{k=2}^n       (-1)^{k-1}  \binom{n-k}{k-1}   C_{n-k+1}  \qquad (n\ge 1). $$
Hence (\ref{recurrenceforan}) can be seen as a $q$-analog of the relation above and the polynomials $B_j^k$ can be viewed as $q$-analogs of the  coefficients $(-1)^{k-1}  \binom{j-k}{k-1}$.

\section{A family of polynomials}  

\subsection{First step}  
Our recurrence relation (\ref{recurrenceBkj})  for the polynomials $B_j^k$ translates into 
\begin{equation}\label{recurrencebkj}
 b_{j}^k   =  b_{j-1}^{k-1}+(1-q^j) \   b_{j-1}^{k}-
 b_{j-2}^{k-1}   \quad (2 \le k \le j-1)      
\end{equation}
for the polynomials $b_j^k= B_j^{j-k+1}$ and the initial conditions in 
Proposition  \ref{propanj}  become:
$$b_{j}^j=1 \text{ for } j\ge 1,  \qquad 
b_{j}^1=   (1-q^2) \cdots (1-q^j)   \text{ for } j\ge 2  .$$  

From (\ref{recurrencebkj})  we have 
$b_{j}^{j-1}  =  b_{j-1}^{j-2}-q^j  $, giving by iteration:   
\begin{equation}\label{bjmoins1j}
  b_{j}^{j-1}   = 1-q^2 - \cdots - q^j =  1- \psi(j-1) 
\end{equation}
where we define  $\psi(k)= q^2 + \cdots + q^{k+1}$ for $k\ge 1$.

Remembering that $  b_j^1=  (1-q^2)  \cdots   (1-q^j)$, and with the help of some computations for small $j$ and $k$, we are led to consider the family $b_j^k$ for fixed $j$ and variable $k$ as a family of polynomials 
intermediate between $ 1-q^2 - \cdots - q^j
$  and    $   (1-q^2)  \cdots   (1-q^j)$. They can be expressed   in terms of the polynomials 
$\Sigma\Pi(a,b) [l_1, l_2, \cdots, l_u]$   
that we define in the next subsection.

\subsection{The polynomials $\Pi$ and $\Sigma\Pi$}\label{notationPi}

We define  $\Pi(u,v)=1$ for $  u > v$  and 
$$ \Pi(u,v)= (1-q^u) (1-q^{u+1}) \cdots (1-q^v) = \prod_{t=u}^v  (1-q^t)  \text{ for } 1\le u \le v  
. $$ 

Starting with a product $\Pi(u,v)$, we make the following transformation attached to $s\ge 2$ and $i$, $u\le i \le v-s+1$: we replace the subproduct 
$(1-q^i)\cdots (1-q^{i+s-1})$  by the monomial $-q^i$. We denote the resulting polynomial by $\Pi(u,v)[i(s)]$: $s$ stands for the number of consecutive terms $(1-q^x)$ suppressed, which we will refer to as the {\it length of the gap}, and $i$ means that the term of lowest degree  suppressed in this gap is $(1-q^i)$, replaced by $-q^i$. For instance: 
$$\Pi(4,10)[5(3)] =  (1-q^4) (-q^5) (1-q^8)(1-q^9)(1-q^{10}).$$  
Now let $s_1, \cdots, s_p$  be integers at least equal to $2$ and let $i_1, \cdots, i_p$ satisfy  
$$  u \le i_1, \     i_t +s_t\le  i_{t+1} \text{ for }
1\le t \le p-1,   \     i_p+s_{p} -1 \le v.  
 $$
The notation 
$$\Pi(u,v)[ i_1(s_1), i_2(s_2), \cdots, i_p(s_p)] 
 $$ 
is almost self-explanatory: we cut a gap of length $s_t$ at $i_t$ and replace it by $-q^{i_t}$.  
The conditions specify that we can have two gaps next to each other but not overlapping.  For instance: 
$$\Pi(4,10)[5(3),8(2)] =  (1-q^4) (-q^5 )(-q^8)(1-q^{10}),   
$$
$$
\Pi(4,10)[5(2),8(3)] =  (1-q^4) (-q^5 )(1-q^7)(- q^8).$$ 

\medskip   

  We write, for $a\le b$ and $l\ge 2$:
$$
\Sigma\Pi(a,b) [l] =  \sum_{i=a}^{b-l+1}  \Pi(a,b) [i(l)]  
$$
and observe that it is the maximal meaningful sum of terms $  \Pi(a,b) [i(l)]  $. Following this observation we notice that, in a term with  gaps at $i_1, \dots, i_u$, we may accept that the gaps 
neighbour each other, but not overlap each other. So if we are to write the maximal meaningful 
sum of such terms,  the relevant information is the  sequence of the lengths of the gaps, 
say $(l_1, l_2, \cdots, l_u)$.  
\begin{definition}\label{SigmaPi}  Let $l_1$, $l_2$, ..., $l_u$ be integers at least equal to $2$. 
We write 
$$\begin{aligned}
\Sigma\Pi(a,b) [l_1, l_2, \cdots, l_u] &=  \sum_{(i_1, i_2, \dots, i_u)\in I} \Pi(a,b) [i_1(l_1), i_2(l_2), \cdots, i_u(l_u)]   \qquad  \text{  where  }    \\
I  =\{ (i_1, i_2, \cdots, i_u) \ /  \  &a\le i_1,   i_t + l_t \le i_{t+1} \text{ for } 1\le t < u, i_u+l_u-1\le b    \}   .   
\end{aligned}$$
We make the usual convention that the value is $0$ if $I$ is empty, which happens 
if and only if  $a + l_1+ \cdots + l_u -1 > b$. 
\end{definition}
\begin{remark}\label{limitSigmaPi}
If $a + l_1+ \cdots + l_u -1 = b$ i.e.  $ l_1+ \cdots + l_u =  b-a+1$, the set $I$ has only one element and 
$$ 
\Sigma\Pi(a,b) [l_1, l_2, \cdots, l_u] = (-1)^u  q^{ua + (u-1)l_1 +  \cdots + l_{u-1}}. 
$$
\end{remark}

\subsection{The polynomials $b_j^k$}

We are ready to write a formula for $b_j^k$.    

\begin{theorem}\label{polbkj}
Let $\psi(k)= q^2 + \cdots + q^{k+1}$ for $k\ge 1$  and let  
   $(b_j^k)_{1\le k \le j}$  be the  unique family of polynomials in the variable $q$ satisfying
the initial conditions 
$$b_{j}^j=1 \text{ for } j\ge 1,  \qquad 
b_{j}^1=   (1-q^2) \cdots (1-q^j)   \text{ for } j\ge 2  $$  
and   recurrence relation (\ref{recurrencebkj}): 
$$
 b_{j}^k   =  b_{j-1}^{k-1}+(1-q^j) \   b_{j-1}^{k}-
 b_{j-2}^{k-1}    \quad (2 \le k \le j-1)      . 
$$
Then $b_j^k$ is given, for $1 \le k \le j-1$, by the following formula: 
\begin{equation}\label{formulabjk} 
\begin{aligned}
b_j^k   &=   
  \sum_{t=1}^{k} (1-\psi(t) )   \  b(j,k,t)  
\end{aligned}
\end{equation}
involving   the polynomials $b(j,k,t)$ defined   for $t=k$ by  
$$b(j,k,k)=    \Pi(k+2,j)   $$ and  for  $1 \le t < k \le j-1$ by 
$$ b(j,k,t)=  
  \sum_{u=1}^{\min\{j-k-1,k-t\}}  
 \sum_{ 
d_1 + \cdots + d_u = k-t    }  
 \Sigma\Pi (t+2,j)[d_1+1, \cdots, d_u+1]    
$$
where the indices $d_1, \cdots, d_u$ are positive integers. 
\end{theorem}

Before proceeding with the proof, we remark that the upper bound in $u$ in the sum defining 
$b(j,k,t)$ may be increased  whenever convenient. Indeed, for $u > k-t$, the sum over $(d_1, \cdots, d_u)$ such that $d_1 + \cdots + d_u=k-t$ is equal to $0$ since the summation set is empty, whereas 
for $u > j-k-1$, the sum of the lengths of the gaps, equal to $u+k-t$, is greater than the available length 
$j-t-1$, so that the corresponding term $\Sigma\Pi$ is equal to $0$ by convention. 

\begin{proof} The formula holds for $k=1$ and  $k=j-1$ (recall (\ref{bjmoins1j})).   The family $(b_j^k)$ is unique by Proposition  \ref{propanj}, and following the proof of this Proposition, it is enough to prove that, 
assuming $ b_{j-1}^{k-1}$,  $ b_{j-1}^{k}$ and $
 b_{j-2}^{k-1}$ are given by (\ref{formulabjk}), then relation  (\ref{recurrencebkj}) implies that 
$b_j^k$ is also given by  (\ref{formulabjk}). We proceed,  taking $k$ such that   $2 \le k \le j-2$.   

We see $b_{j-1}^{k-1}$, $  b_{j-1}^{k}$ and $
 b_{j-2}^{k-1} $ as sums of $k$ or $k-1$ terms and write accordingly $b_j^k$ as a sum of $k$ terms, namely 
$$\begin{aligned}
b_j^k   
&=    \sum_{t=1}^{k} (1-\psi(t) ) \beta(j,k,t)  
\end{aligned}
$$
where we let 
$$ \beta(j,k,t)=   b(j-1,k-1,t) + (1-q^j)  b(j-1,k,t)  - b(j-2,k-1,t)$$  
and agree on $b(x,y,z)= 0$ if $z>y$.  

We first note that 
$$ \beta(j,k,k)=      (1-q^j)  \Pi(k+2,j-1) =    \Pi(k+2,j) = b(j,k,k).$$
We look next at the terms   in $t=k-1$.  
$$ \begin{aligned}
\beta(j,k,k-1)&=     \Pi(k+1,j-1)  -   \Pi(k+1,j-2)   \\&  
 +  (1-q^j)     \sum_{u=1}^{\min\{j-k-2,1\}}  
 \sum_{ 
d_1 + \cdots + d_u = 1   }  
 \Sigma\Pi (k+1,j-1)[d_1+1, \cdots, d_u+1]  
  .   
\end{aligned}
$$
  If $k=j-2$ the middle term is $0$ and we have  
$$     \Pi(j-1,j-1)  -     \Pi(j-1,j-2)=  -q^{j-1}  = \Sigma\Pi (j-1,j)[2 ] = b(j,j-2,j-3)  . 
$$ 
  If $k<j-2$ we claim that 
$$ \begin{aligned}     
&  \Pi(k+1,j-1)  
 +  (1-q^j)      
 \Sigma\Pi (k+1,j-1)[2 ] 
-    \Pi(k+1,j-2) = \Sigma\Pi (k+1,j)[2 ]      .   
\end{aligned}
$$
Indeed:  $
\Sigma\Pi(k+1,j) [2] =  \sum_{i=k+1}^{j-1}  \Pi(k+1,j) [i(2)]  
$.  
In this sum, the terms for $i<j-1 $ end with a $(1-q^j)  $ and their sum is equal to $ (1-q^j)      
 \Sigma\Pi (k+1,j-1)[2 ] $.   The last term, for $i=j-1$, is equal to 
$-q^{j-1}   \Pi(k+1,j-2)$ which is the sum of the first and third terms above, q.e.d. We get 
$$\beta(j,k,k-1) = b(j,k,k-1).$$ 

For $t < k-1$ we don't have such an equality, nonetheless we compute $\beta(j,k,t)$. We use the remark preceeding the proof to simplify the upper bound in $u$. 
\begin{equation}\label{beta1}    \begin{aligned}     
   \beta(j,k,t)  =
&     \sum_{u=1}^{k-t}  
 \sum_{ 
d_1 + \cdots + d_u = k-1-t    }  
 \Sigma\Pi (t+2,j-1)[d_1+1, \cdots, d_u+1]   \\
+(1-q^j)   & \sum_{u=1}^{k-t}  
 \sum_{ 
d_1 + \cdots + d_u = k-t    }  
 \Sigma\Pi (t+2,j-1)[d_1+1, \cdots, d_u+1]   \\
 -   & \sum_{u=1}^{k-t}  
 \sum_{ 
d_1 + \cdots + d_u = k-1-t    }  
 \Sigma\Pi (t+2,j-2)[d_1+1, \cdots, d_u+1] .
\end{aligned}
\end{equation}

Let $d=(d_1, \cdots, d_u)$ and, following  Definition \ref{SigmaPi},  write $I_s(d)$ for the summation set of $\Sigma\Pi (t+2,s)[d_1+1, \cdots, d_u+1] $. We have: 
$$
I_s(d)  =\{ (i_1, i_2, \cdots, i_u) \ /  \  t+2\le i_1,   i_t + d_t +1\le i_{t+1} \text{ for } 1\le t < u, i_u+d_u\le s    \}   .  
$$
Hence 
$I_{s-1}(d)$ is contained in $I_s(d)$, with complement 
$$
X_s(d) =\{ (i_1, i_2, \cdots, i_u) \ /  \  t+2\le i_1,   i_t + d_t +1\le i_{t+1} \text{ for } 1\le t < u, i_u+d_u = s     \}  .     
$$ 
We can write
\begin{equation}\label{equationXs}     \begin{aligned}   
 \Sigma\Pi (t+2,s)[d_1+1, \cdots, &d_u+1]  =  \\    \  (1-q^s)   \   
 &\Sigma\Pi (t+2,s-1)[d_1+1, \cdots, d_u+1]  
\\    +  
  \sum_{(i_1, i_2, \dots, i_u)\in X_s(d)}  & \Pi(t+2,s) [i_1(d_1+1),  \cdots, i_u(d_u+1)] .   
\end{aligned}  \end{equation}
We analyse the sum over $X_s(d)$ according to the value of $d_u$.
\begin{itemize}
\item  If $d_u=1$, we let 
$d^-=(d_1, \cdots, d_{u-1})$. We have: 
$$  \begin{aligned}  
 \sum_{(i_1, i_2, \dots, i_u)\in X_s(d)}  & \Pi(t+2,s) [i_1(d_1+1),   \cdots, i_u(d_u+1)] \\
= &-q^{s-1}   \sum_{(i_1, i_2, \dots, i_{u-1})\in I_{s-2}(d^-)}   \Pi(t+2,s-2) [i_1(d_1+1),   \cdots, i_{u-1}(d_{u-1}+1)] \\
=   &-q^{s-1}   \     \Sigma\Pi(t+2,s-2) [d_1+1,   \cdots, d_{u-1}+1].   
\end{aligned} $$
This equality holds even for $u=1$ if we consider this last expression as meaning $-q^{s-1}   \  \Pi(t+2,s-2)$ if $u=1$.   

\item 
If $d_u > 1$, we let  $d'=(d_1, \cdots, d_{u-1}, d_u-1)$. The condition $(i_1, i_2, \dots, i_u)\in X_{s}(d)$  is equivalent to the condition $(i_1, i_2, \dots, i_u)\in X_{s-1}(d')$. We get
$$  \begin{aligned}  
 \sum_{(i_1, i_2, \dots, i_u)\in X_s(d)}  & \Pi(t+2,s) [i_1(d_1+1),   \cdots, i_u(d_u+1)] \\
= & \sum_{(i_1, i_2, \dots, i_u)\in X_{s-1}(d')}  \Pi(t+2,s-1) [i_1(d'_1+1),   \cdots, i_u(d'_u+1)].   
\end{aligned} $$

\end{itemize}

We sum up (\ref{equationXs})   over $(d_1, \cdots, d_u)$, for $s=j$: 
\begin{equation}\label{sumupXs}
 \begin{aligned}   
   (1-q^j)   &
 \sum_{ 
d_1 + \cdots + d_u = k-t    }  
 \Sigma\Pi (t+2,j-1)[d_1+1, \cdots, d_u+1]    
 \\
= &  
 \sum_{ 
d_1 + \cdots + d_u = k-t    }  
 \Sigma\Pi (t+2,j)[d_1+1, \cdots, d_u+1]    \\
&-    \sum_{ \begin{smallmatrix}
d_1 + \cdots + d_u = k-t   \cr d_u=1 \end{smallmatrix} }   
 \sum_{(i_1, i_2, \dots, i_u)\in X_j(d)}  \Pi(t+2,j) [i_1(d_1+1),  \cdots, i_u(d_u+1)]
 \\
&-    \sum_{ \begin{smallmatrix}
d_1 + \cdots + d_u = k-t   \cr d_u>1 \end{smallmatrix} }   
 \sum_{(i_1, i_2, \dots, i_u)\in X_j(d)}  \Pi(t+2,j) [i_1(d_1+1),  \cdots, i_u(d_u+1)] 
 \\
=   &
 \sum_{ 
d_1 + \cdots + d_u = k-t    }  
 \Sigma\Pi (t+2,j)[d_1+1, \cdots, d_u+1]    \\
&+    q^{j-1}     \sum_{ 
d_1 + \cdots + d_{u-1} = k-t -1   }   
 \     \Sigma\Pi(t+2,j-2) [d_1+1,   \cdots, d_{u-1}+1]
 \\
-    &\sum_{  
d_1 + \cdots + d_u = k-t-1  }   
 \sum_{(i_1, i_2, \dots, i_u)\in X_{j-1}(d)}  \Pi(t+2,j-1) [i_1(d_1+1),  \cdots, i_u(d_u+1)] 
    .   
\end{aligned}
\end{equation}
In these equalities we have $u\le k-t$, otherwise all sums are $0$, and if $u=k-t$ the last line is $0$. For    $u\le k-t$, the term in $u$ in $\beta(j,k,t)$ is then  equal to:   
\begin{equation}\label{sumupXs2}
 \begin{aligned}   
 &  \sum_{ 
d_1 + \cdots + d_u = k-1-t    }  \left[  
 \Sigma\Pi (t+2,j-1)[d_1+1, \cdots, d_u+1]   \right.   \\  & \left. 
\qquad \qquad   \qquad  \qquad \qquad   \qquad   
- \Sigma\Pi (t+2,j-2)[d_1+1, \cdots, d_u+1] \right]  \\
&   +    
 \sum_{ 
d_1 + \cdots + d_u = k-t    }  
 \Sigma\Pi (t+2,j)[d_1+1, \cdots, d_u+1]    \\
&+    q^{j-1}     \sum_{ 
d_1 + \cdots + d_{u-1} = k-t -1   }   
 \     \Sigma\Pi(t+2,j-2) [d_1+1,   \cdots, d_{u-1}+1]
 \\
-    &\sum_{  
d_1 + \cdots + d_u = k-t-1  }   
 \sum_{(i_1, i_2, \dots, i_u)\in X_{j-1}(d)}  \Pi(t+2,j-1) [i_1(d_1+1),  \cdots, i_u(d_u+1)] 
    .   
\end{aligned}
\end{equation}
Using (\ref{equationXs}) again we see that the first, second  and last line add up to 
$$  - q^{j-1}   
\sum_{  
d_1 + \cdots + d_u = k-t-1  } \Sigma\Pi (t+2,j-2)[d_1+1, \cdots, d_u+1]
$$
We can now  compute $\beta(j,k,t)$.
$$  \begin{aligned}     
   \beta(j,k,t)  =
&     \sum_{u=1}^{ k-t}  \left[      - q^{j-1}   
\sum_{  
d_1 + \cdots + d_u = k-t-1  }   \Sigma\Pi (t+2,j-2)[d_1+1, \cdots, d_u+1]  \right. 
\\ 
& \qquad  +   
 \sum_{ 
d_1 + \cdots + d_u = k-t    }  
 \Sigma\Pi (t+2,j)[d_1+1, \cdots, d_u+1]    \\
&   \qquad  +  \left.    q^{j-1}     \sum_{ 
d_1 + \cdots + d_{u-1} = k-t -1   }   
 \     \Sigma\Pi(t+2,j-2) [d_1+1,   \cdots, d_{u-1}+1]    \right] 
  \\
=&  \     b(j,k,t) 
\\
&+  q^{j-1}   \left[    -   \sum_{u=1}^{ k-t-1}    \sum_{  
d_1 + \cdots + d_u = k-t-1  }   \Sigma\Pi (t+2,j-2)[d_1+1, \cdots, d_u+1]  \right. 
\\   &   \qquad  \qquad +    \left.     \sum_{u=0}^{ k-t-1}       \sum_{ 
d_1 + \cdots + d_{u} = k-t -1   }   
 \     \Sigma\Pi(t+2,j-2) [d_1+1,   \cdots, d_{u}+1]    \right]    .
\end{aligned}
$$
The two last lines cancel one another, except for  ``$u=0$'', for which we use our convention below 
 (\ref{equationXs}): those terms come from  terms with  $u=1$ and  $d_u=d_1= k-t=1 $ which only happens if $t=k-1$, a case treated separately before. Here we have $t\le k-2$, 
  giving finally:  
$$   \begin{aligned}     
   \beta(j,k,t)  
=&  \     b(j,k,t) 
     .
\end{aligned}
$$

We obtain:  
$$\begin{aligned}
b_j^k   
&=    \sum_{t=1}^{k} (1-\psi(t) ) \beta(j,k,t)   \\
&=  (1-\psi(k) ) b(j,k,k )  +  (1-\psi(k-1) ) b(j,k,k-1 )  \\
&    \qquad      +  \sum_{t=1}^{k-2} (1-\psi(t) )
 b(j,k,t) 
\end{aligned}
$$
 which is exactly (\ref{formulabjk} ).  
\end{proof}

\medskip
\begin{remark}\label{recurrenceforeacht} 
We have actually  proved that the family of polynomials  $b(j,k,t)$ defined in the Theorem
for $1\le t \le k \le j-1$  and   
extended by $b(x,y,z)= 0$ if  $x>y$ and $z>y$,  
satisfies the following recurrence relation, for $2 \le k \le j-2$:    
$$
 b(j,k,t)=   b(j-1,k-1,t) + (1-q^j)  \  b(j-1,k,t)  - b(j-2,k-1,t)       .   
 $$

\end{remark}  


\section{Conclusion and further questions}

\subsection{A general formula}

We established the following in the course of our computations. We do not know if it is a known formula, in any case we couldn't find a reference for it. We point out that in the Proposition below, we make no assumption on the double sequence $b_n^i$, in other words, in this subsection and the next, the notation $b_n^i$ does not  refer to the polynomials that we have computed  in the previous  section, nor does $\psi$ refer to the function used before. 

\begin{proposition}\label{generalprop2}
Let $A$ be a commutative ring. Let 
 $n \mapsto \psi(n)$, $n\ge 1$, be a  function from $\mathbb N$ to $A$ and let  
$(b_n^i)_{ n \ge 2,   1 \le i \le n-1 }$  be a  double sequence of elements in $A$. 
Let  $(u_n)_{n \ge 1}$   be the sequence of elements in $A$ defined by 
$u_1= \psi(1)$ and the following recurrence relation: 
\begin{equation}\label{newnotationrecurrence}
u_n =  \psi(n) - \sum_{i=1}^{n-1}  \    b_n^i \    u_i  \qquad    (n \ge 2).    
\end{equation}
 
  Then we have, for $n\ge 1$:  

\begin{equation}\label{generalan2newnotation}
  u_n=  \psi(n) + \   \sum_{i=1}^{n-1} \psi(i) \ \sum_{s=0}^{n-1-i}      (-1)^{s+1}     \  
\sum_{n= v_0  >   \cdots > v_{s} > v_{s+1}=i  }    \    \prod_{u=0}^{s} 
b_{v_u}^{v_{u+1}} .    
\end{equation} 
\end{proposition} 

\medskip 

 \begin{proof} This holds for $n=1$. We assume it holds for any $m<n$ and prove it for $n$. 
$$\begin{aligned}
u_n&=  \psi(n) - \sum_{i=1}^{n-1}  \    b_n^i \    u_i  \\
&=  \psi(n) - \sum_{i=1}^{n-1}  \    b_n^i \    \left( \psi(i) + \   \sum_{j=1}^{i-1} \psi(j) \ \sum_{s=0}^{i-1-j}      (-1)^{s+1}     \  
\sum_{i= v_0  >   \cdots > v_{s}  > v_{s+1}= j }    \    \prod_{u=0}^{s} 
b_{v_u}^{v_{u+1}} \right)  .
\end{aligned}$$ 
We exchange the sums in $i$ and $j$; a  term $\psi(j)$ appears in the $i$-th term if and only if 
$i \ge j$, whence: 
$$\begin{aligned}
u_n&=     \psi(n) - \sum_{j=1}^{n-1}  \psi(j)  \left(  \    b_n^j   + \ 
  \sum_{i=j+1}^{n-1} b_n^i \ \sum_{s=0}^{i-1-j}      (-1)^{s+1}     \  
\sum_{i= v_0  >   \cdots > v_{s}  > v_{s+1}= j }    \    \prod_{u=0}^{s} 
b_{v_u}^{v_{u+1}}\right)   \\
&=   \psi(n) + \sum_{j=1}^{n-1}  \psi(j)  \left(  \  -  b_n^j   + \ 
  \sum_{i=j+1}^{n-1} \ \sum_{s=0}^{i-1-j}      (-1)^{s+2}     \  
\sum_{i= v_0  >   \cdots > v_{s}  > v_{s+1}= j }    \   b_n^i  \prod_{u=0}^{s} 
b_{v_u}^{v_{u+1}} \right) .
\end{aligned}$$ 
We change the index $s$ into $s-1$ with $1\le s \le i-j$, then shift the indices in the last sum to let 
in $v_0=n$:    
 $$\begin{aligned}
u_n&= 
  \psi(n) + \sum_{j=1}^{n-1}  \psi(j)  \left(  \  -  b_n^j   + \ 
  \sum_{i=j+1}^{n-1} \ \sum_{s=1}^{i-j}      (-1)^{s+1}     \  
\sum_{n=v_0>i= v_1  >   \cdots > v_{s}  > v_{s+1}= j }    \    \prod_{u=0}^{s} 
b_{v_u}^{v_{u+1}} \right) 
 \\
&=   \psi(n) + \sum_{j=1}^{n-1}  \psi(j)  \left(  \  -  b_n^j   +  \ \sum_{s=1}^{n-1-j}    (-1)^{s+1}      \ 
  \sum_{i=j+s}^{n-1}     \  
\sum_{n=v_0>i= v_1  >   \cdots > v_{s}  > v_{s+1}= j }    \    \prod_{u=0}^{s} 
b_{v_u}^{v_{u+1}} \right)     
\end{aligned}$$ 
where the coefficient of 
$\psi(j)$ is exactly 
$$\sum_{s=0}^{n-1-j}      (-1)^{s+1}     \  
\sum_{n= v_0  >   \cdots > v_{s} > v_{s+1}=j  }    \    \prod_{u=0}^{s} 
b_{v_u}^{v_{u+1}}      $$  as expected.  
\end{proof}

\subsection{Matrix interpretation}\label{matricesnewnotation}

Recurrence relation (\ref{newnotationrecurrence}) can be viewed in matrix form as: 
$$
P  U = \Psi   \qquad   \text{ i.e. }     U = P^{-1}  \Psi 
$$
with  
$$U=\left( \begin{matrix}u_1 \cr u_2 \cr \vdots  \cr u_n \cr \vdots \end{matrix}\right), \quad  
P=\left( \begin{matrix} 1 & 0 & \cdots & \cdots& \cdots   \cr 
b_2^1 & 1 &0& \cdots  & \cdots   \cr  \vdots & \ddots &\ddots  & \ddots& \cdots \cr 
b_{n}^{1} &   \cdots & b_{n}^{n-1}& 1 & \cdots \cr 
 \vdots & \vdots &\vdots  & \vdots& \vdots  \end{matrix}\right), \quad
 \Psi=\left( \begin{matrix} \psi(1) \cr \psi(2) \cr \vdots \cr \psi(n) \cr \vdots \end{matrix}\right).
$$
We also  write 
$$ P^{-1}=\left( \begin{matrix} 1 & 0 & \cdots & \cdots& \cdots   \cr 
c_2^1 & 1 &0& \cdots  & \cdots   \cr  \vdots & \ddots &\ddots  & \ddots& \cdots \cr 
c_{n}^{1} &   \cdots & c_{n}^{n-1}& 1 & \cdots \cr 
 \vdots & \vdots &\vdots  & \vdots& \vdots  \end{matrix}\right)$$
and obtain 
$$
u_n=   \psi(n) + \   \sum_{i=1}^{n-1}   c_n^i \psi(i) =  \   \sum_{i=1}^{n}   c_n^i \psi(i). 
$$
The $k$-th column $Q_k$  of $ P^{-1}$  satisfies $PQ_k= L_k  $, where $L_k$ is  the $k$-th column of the identity matrix, so by the general formula  (\ref{generalan2newnotation}) we have 
$$
(Q_k)_n =  (L_k)_n  +  \sum_{i=1}^{n-1} (L_k)_i \ \sum_{s=0}^{n-1-i}      (-1)^{s+1}     \  
\sum_{n= v_0  >   \cdots > v_{s} > v_{s+1}=i   }    \    \prod_{u=0}^{s} 
b_{v_u}^{v_{u+1}} 
$$
Indeed for $n<k$ this formula yields $0$, for $n=k$ it yields $1$,  and for $n>k$ it yields 
$$
c_n^k=   \sum_{s=0}^{n-1-k}      (-1)^{s+1}     \  
\sum_{n= v_0  >   \cdots > v_{s}  > v_{s+1}=k }    \    \prod_{u=0}^{s} 
b_{v_u}^{v_{u+1}} .
$$ 
 
In other words, formula (\ref{generalan2newnotation}) amounts to the calculation of the inverse of a triangular matrix, which is most likely  known,

\subsection{The Poincar\'e polynomial for $W^c(A_n)$}    
We now apply 
 Proposition \ref{generalprop2}   to our sequence $(a_n)_{n\ge 0}$.  
The dictionary between    (\ref{recurrenceforan}), multiplied by $q$ for convenience, and 
(\ref{newnotationrecurrence})
is: 
$$
u_n= q^{n+1} a_{n-1},  \quad  \psi(n)= q^2 + \cdots + q^{n+1}, \quad   
b_n^i = B_n^{n-i+1}.   
$$ 
We get immediately formula (\ref{alternativean}) in the following Theorem: 

\begin{theorem}\label{maintheorem}
For $1\le k  \le j$ we let $(b_j^k)_{1\le k \le j}$ be  the unique family of polynomials described in Theorem \ref{polbkj}.  
The Poincar\'e polynomial $a_n(q)$ is given, for $n \ge 1$,  by the following formula:  
\begin{equation}\label{alternativean}
\begin{aligned}
 q^{n+2} a_n=  &q^2 + \cdots + q^{n+2}   \\  & + \   \sum_{i=1}^{n} \  (q^2+ \cdots +q^{i+1})   
   \ \sum_{s=0}^{n-i}      (-1)^{s+1}   
\sum_{n+1= v_0  >   \cdots > v_{s}  > v_{s+1}=i  }  \    \prod_{u=0}^{s} 
b_{v_u}^{v_{u+1}}   .  
\end{aligned}
\end{equation}

\end{theorem}

\subsection{A slightly shorter formula}

We use the notation in subsection \ref{matricesnewnotation} and consider the first column 
$(c_i^1)_{i\ge 1}$    of the inverse matrix of   $P= (b_i^j)$ (notation extended to  $b_i^j=0$ for $j >i$). 
In the case of the Poincar\'e polynomial, 
there is actually a shortcut in the computation:
 \begin{proposition}\label{Cnnandan}  
$$
c_{n+2}^1=  -(1-q^2) u_{n} =   -(1-q^2)  q^{n+1} a_{n-1} \quad \text{ for } n \ge 1  .  
$$
\end{proposition}
As in subsection \ref{matricesnewnotation},  
we  get a matrix form of this identity as follows: 
$$
S Q_1= S P^{-1} L_1=  -(1-q^2)   U = -(1-q^2)  P^{-1} \Psi  
$$
where $S$ is the matrix of the double shift: 
$$ S=\left( \begin{matrix} 0 & 0 & 1 &0 & \cdots   \cr 
0 & 0 &0&1& \ddots     \cr  \vdots & \ddots &\ddots  & \ddots& \ddots \cr 
\vdots &  &\ddots  &\ddots& \ddots\cr 
 \vdots &   &   &\ddots& \ddots  \end{matrix}\right),  \   SU= \left( \begin{matrix}u'_1 \cr u'_2 \cr \vdots  \cr u'_n \cr \vdots \end{matrix}\right),  u'_i=u_{i+2}. $$
This amounts to  $PSP^{-1} L_1=  -(1-q^2)  \Psi  $, that is, the first column of  $PSP^{-1} $ 
is  $-(1-q^2)  \Psi  $. This  formulation is equivalent to the Proposition  and  gives relations involving the $b_j^k$. We write the generic one: 

\begin{equation}\label{samewithcn} 
\sum_{k=1}^n  b_n^k   c_{k+2}^1 
=  -(1-q^2) \psi(n) \end{equation}
i.e.  
$$\begin{aligned}  
 (1-q^2) \psi(n)  &=  -  \sum_{k=1}^n \     b_n^k   \  
  \sum_{s=0}^{k}      (-1)^{s+1}     \  
\sum_{k+2= v_0  >   \cdots > v_{s}  > v_{s+1}=1 }    \    \prod_{u=0}^{s} 
b_{v_u}^{v_{u+1}} .
\end{aligned}
$$

We proceed  to the proof of the Proposition. 
\begin{proof}
Since the sequence $( -(1-q^2)u_n)_{n\ge 1}$ is uniquely determined by relation  
\begin{equation}\label{newnotationrecurrencemultipliedaa}
 -(1-q^2)u_n = -(1-q^2)  \psi(n) -\sum_{i=1}^{n-1}  \    b_n^i \   (-(1-q^2) u_i ) \qquad    (n \ge 2)    
\end{equation}
as in  (\ref{newnotationrecurrence}), 
together with  the first term 
$ -(1-q^2)u_1=  -(1-q^2) q^2$, 
it is enough to prove that the sequence $(c_{n+2}^1)_{n\ge 1}$ satisfies the same conditions, as announced in (\ref{samewithcn}). 

We check the first term. Indeed: 
$$\begin{aligned}  
c_3^1&=    \sum_{s=0}^{1}      (-1)^{s+1}     \  
\sum_{3= v_0  >   \cdots > v_{s}  > v_{s+1}=1 }    \    \prod_{u=0}^{s} 
b_{v_u}^{v_{u+1}}   \\
&=   - b_3^1 +  b_3^2 b_2^1 \\
&=  -(1-q^2)(1-q^3) + (1-q^2-q^3)(1-q^2)  \\
&=  -q^2(1-q^2). 
\end{aligned}$$
Now we must show  (\ref{newnotationrecurrencemultipliedaa}) for $(c_{n+2}^1)_{n\ge 1}$, namely: 
\begin{equation}\label{newnotationrecurrencemultiplied}
c_{n+2}^1 = -(1-q^2)  \psi(n) -\sum_{i=1}^{n-1}  \    b_n^i \  c_{i+2}^1 \qquad    (n \ge 2)  , 
\end{equation}
whereas $(c_{n+2}^1)_{n\ge 1}$ satisfies 
$
b_{n+2}^1 + \sum_{i=2}^{n+1} b_{n+2}^i  c_i^1 + c_{n+2}^1=0
$
i.e. 
\begin{equation}\label{initialwithcn} 
 c_{n+2}^1=-b_{n+2}^1 -  \sum_{i=2}^{n+1} b_{n+2}^i  c_i^1   . 
\end{equation}
A main difference between the two relations is that the coefficient of  $c_i^1$ in  (\ref{newnotationrecurrencemultiplied}) is  $b_n^{i-2}$, instead of $ b_{n+2}^i $  in (\ref{initialwithcn}).  
We thus use recurrence relation (\ref{recurrencebkj})   on  $b_j^k$ to drop the indices, and we use it with an extended range  of    values: 
it actually holds for $k=1$ provided we set $b_j^0=0$  for $j \ge 0$.      
We replace in (\ref{initialwithcn}   )   $b_{n+2}^i$ by   
$$  b_{n+2}^i   =(1-q^{n+2}) \   b_{n+1}^{i} +   b_{n+1}^{i-1}   -
 b_{n}^{i-1}   \quad (1 \le i \le n+1)    $$  
and gather first all terms coming from the first one above, the one with a factor   $ (1-q^{n+2})$. 
We get   the product of   $ (1-q^{n+2})$ and   
$$ -b_{n+1}^1 -  \sum_{i=2}^{n+1} b_{n+1}^i  c_i^1 = 0,   
$$
since the term  for $i=n+1$ in the sum is $-c_{n+1}^{1}$, while the rest is  equal to $c_{n+1}^{1}$. 
Hence we can replace in (\ref{initialwithcn}   )   $b_{n+2}^i$ by    $b_{n+1}^{i-1}   -
 b_{n}^{i-1}$, getting:  
\begin{equation}\label{recCNNtechnical} 
c_{n+2}^1= -  \sum_{i=2}^{n+1} (b_{n+1}^{i-1}   -
 b_{n}^{i-1}) c_i^1   . 
\end{equation}
We use once again  (\ref{recurrencebkj})    in the following form: 
$$
b_{n+1}^{i-1}   -
 b_{n}^{i-1}  =  -q^{n+1} \   b_{n}^{i-1} +   b_{n}^{i-2}   -
 b_{n-1}^{i-2}   \quad (1 \le i -1  \le n)  
$$
and replace this in the previous expression: 
\begin{equation}\label{expressionwithPhi}
\begin{aligned}  
c_{n+2}^1&= -  \sum_{i=2}^{n+1} ( -q^{n+1} \   b_{n}^{i-1} +   b_{n}^{i-2}   -
 b_{n-1}^{i-2} ) c_i^1  =   \Phi    -  \sum_{i=3}^{n+1}   b_{n}^{i-2}   c_i^1    
\end{aligned}
\end{equation} (recall $b_n^0=0$), 
letting $\Phi= -  \sum_{i=2}^{n+1} ( -q^{n+1} \   b_{n}^{i-1}     -
 b_{n-1}^{i-2} ) c_i^1  . 
$
We compute $\Phi$.

\begin{lemma}\label{lemmaPhi}
For any $k$, $0  \le k  \le  n-1$, we define  
$$  \Phi  (k) 
=   \sum_{u=k+2}^{n+1}  ( (q^{n+1}+ \cdots + q^{n-k+1}) b_{n-k}^{u-k-1} +    b_{n-k-1}^{u-k-2} )  
\    c_{u-k}^{1}.  
$$ 
Then for any $k$, $0 \le k \le  n-1$, we have $\Phi = \Phi(k)$.   Consequently 
$$\Phi= \Phi(n-1)= - (1-q^2) (q^{n+1}+ \cdots + q^{2} )   = -(1-q^2)  \psi(n)   .$$  
\end{lemma} 
 
\begin{proof}
The statement for $k=0$ is just the definition of $\Phi$. The conclusion  
comes with $c_2^1= -b_2^1=-(1-q^2) $. It remains to  take some $k$, 
$0\le k  \le n-2$, such that  $\Phi = \Phi(k)$  and  prove that  $\Phi = \Phi(k+1)$.

Indeed, the term with $u=n+1$ in $\Phi(k)$ is, using  (\ref{recCNNtechnical}):   
$$ \begin{aligned} 
( q^{n+1} + \cdots &+ q^{n-k+1} +1)    c_{n-k+1}^{1}  =  -  
( q^{n+1}+  \cdots + q^{n-k+1} +1)   
   \sum_{s=2}^{n-k}  (b_{n-k}^{s-1}   -
 b_{n-k-1}^{s-1})  c_{s}^{1} 
\end{aligned}  $$  
so that 
$$ \begin{aligned} 
\Phi  (k) 
=  & \sum_{u=k+2}^{n}  ( (q^{n+1}+ \cdots + q^{n-k+1}) b_{n-k}^{u-k-1} +    b_{n-k-1}^{u-k-2} )  
\    c_{u-k}^{1}  \\
& -  
( q^{n+1}+  \cdots + q^{n-k+1} +1)   
   \sum_{s=2}^{n-k}  (b_{n-k}^{s-1}   -
 b_{n-k-1}^{s-1})  c_{s}^{1}   \\ 
=  & \sum_{u=k+2}^{n}  ( - b_{n-k}^{u-k-1}+    b_{n-k-1}^{u-k-2}  + (q^{n+1}+ \cdots + q^{n-k+1}+1) b_{n-k-1}^{u-k-1} )  
\    c_{u-k}^{1}  
\end{aligned}  $$
Relation (\ref{recurrencebkj})  now does the trick, since 
$$
 - b_{n-k}^{u-k-1}=- b_{n-k-1}^{u-k-2}  -  (1-q^{n-k}) \   b_{n-k-1}^{u-k-1}   + b_{n-k-2}^{u-k-2}
$$
giving $$ \begin{aligned} 
\Phi  (k)   =  & \sum_{u=k+2}^{n}  (   (q^{n+1}+ \cdots + q^{n-k}) b_{n-k-1}^{u-k-1}  + b_{n-k-2}^{u-k-2})  
\    c_{u-k}^{1}  
\end{aligned}  $$
which is exactly   $  \Phi(k+1)$.   Lemma \ref{lemmaPhi} is proved.  
\end{proof}
We obtain (\ref{newnotationrecurrencemultiplied}) by replacing 
$\Phi$ by its value $ -(1-q^2)  \psi(n) $ in (\ref{expressionwithPhi}).
\end{proof}  
  
 Proposition   \ref{Cnnandan}  may shorten computer implementations since it decreases indices by 2. 
The method of proof can be iterated but we don't pursue this here.

\subsection{Extensive formulas}\label{extensiveformula}

To conclude this work, we write again the formula for the Poincar\'e  polynomial in an  extensive form,  plugging in the value of  the   $b_j^k$. 

\begin{corollary}\label{extensiveformulacorollary}

The Poincar\'e polynomial $a_n(q)$ is given, for $n \ge 1$,  by the following formula:  
$$
\begin{aligned}
 a_n=  &  \   \frac{1}{ q^{n+2}} \  (q^2 + \cdots + q^{n+2} )    \\   +  &  \  \frac{1}{ q^{n+2}} \   \sum_{i=1}^{n} \  (q^2+ \cdots +q^{i+1})   
   \ \sum_{s=0}^{n-i}      (-1)^{s+1}   
\sum_{n+1= v_0  >   \cdots > v_{s}  > v_{s+1}=i  }  \   \\
 \prod_{u=0}^{s}   \    
&\Bigg[(1-\psi(v_{u+1}) )   \    \Pi(v_{u+1}+2,v_u)   \\
+ \sum_{t=1}^{v_{u+1}-1} &(1-\psi(t) )     \sum_{r=1}^{\min\{v_{u}-v_{u+1}-1,v_{u+1}-t\}}  
  \!  \!  \sum_{d_1 + \cdots + d_r = v_{u+1}-t}  
 \Sigma\Pi (t+2,v_{u})[d_1+1, \cdots, d_r+1]     \Bigg]  
\end{aligned}
$$
where  $\psi(k)= q^2 + \cdots + q^{k+1}$ for $k\ge 1$,   
$ \Pi(a,b) = \prod_{a\le i \le b}(1-q^i)$ for $1\le a \le b$ and 
$\Sigma\Pi (a,b)[e_1, \cdots,  e_u]  $ is the sum of all possible products obtained from $ \Pi(a,b) $ 
by removing   $e_i$ consecutive terms $(1-q^j)\cdots (1-q^{j+e_i-1})$ 
and replacing them by $(-q^j)$, for $1 \le i \le u$. 
\end{corollary}
 
For the sake of completeness we also write the formula that arises from Proposition 
\ref{Cnnandan}, since $ c_{n+3}^1 = - {q^{n+2}(1-q^2)} a_n$. 
\begin{corollary}\label{extensiveformulacorollary2}

The Poincar\'e polynomial $a_n(q)$ is given, for $n \ge 1$,  by the following formula:  
$$
\begin{aligned}
 a_n=  &  \   -\frac{1}{q^{n+2}(1-q^2)}   \  
\sum_{s=0}^{n+1}      (-1)^{s+1}     \  
\sum_{n+3= v_0  >   \cdots > v_{s}  > v_{s+1}=1 }    \    \prod_{u=0}^{s} b_{v_u}^{v_{u+1}}
 \\   =     &  \   -\frac{1}{q^{n+2}(1-q^2)}  \  
\sum_{s=0}^{n+1}      (-1)^{s+1}     \  
\sum_{n+3= v_0  >   \cdots > v_{s}  > v_{s+1}=1 }    \    \prod_{u=0}^{s}   \\   
&\Bigg[(1-\psi(v_{u+1}) )   \    \Pi(v_{u+1}+2,v_u)   \\
+ \sum_{t=1}^{v_{u+1}-1} &(1-\psi(t) )     \sum_{r=1}^{\min\{v_{u}-v_{u+1}-1,v_{u+1}-t\}}  
  \!  \!  \sum_{d_1 + \cdots + d_r = v_{u+1}-t}  
 \Sigma\Pi (t+2,v_{u})[d_1+1, \cdots, d_r+1]     \Bigg]  
\end{aligned}
$$

\end{corollary}

\renewcommand{\refname}{REFERENCES}


\begin{thebibliography}{} 


\bibitem{A2016} S.Al Harbat, Tower of fully commutative elements of type $\tilde {A}$ and applications,   J. Algebra 465 (2016),  111--136.  


\bibitem{Barcucci2001}
E. Barcucci, A. Del Lungo, E. Pergola, and R. Pinzani, Some permutations with forbidden subsequences
and their inversion number. Discrete Math., 234(1-3):1--15, 2001.  

\bibitem{Biagioliandal}
R. Biagioli,  M. Bousquet-M\'elou, F. Jouhet, P.  Nadeau,  
Length enumeration of fully commutative elements in finite and affine Coxeter groups. 
J. Algebra 513 (2018), 466--515. 

\bibitem{BB} A.Bjorner and F.Brenti,   Combinatorics of Coxeter groups.  
GTM 231,
 Springer, 2005.    

\bibitem{Bourbaki_1981} N. Bourbaki,  Groupes et alg\`ebres de Lie, Chapitres 4, 5 et  6,  Masson, Paris, 1981. \label{Bourbaki_1981}


\bibitem{Graham} J. J. Graham, Modular representations of Hecke algebras and related algebras, Ph.D. Thesis, University of Sydney, 1995.


\bibitem{HJ} C. R. H. Hanusa and  B. C. Jones,  The enumeration of fully
commutative affine permutations,  European J. Combin. 31 (5) (2010),  
1342--1359.



\bibitem{Humphreys} J. E. Humphreys,   Reflection groups and Coxeter groups,
Cambridge University Press, volume 29, 1992. 


\bibitem{Jones_1986} V. F. R. Jones,  
Braid groups, Hecke algebras and type II1 factors. Geometric methods in operator algebras (Kyoto, 1983), 242--273,
Pitman Res. Notes Math. Ser., 123, Longman Sci. Tech., Harlow, 1986. 



\bibitem{Stanley} R.P. Stanley,   Catalan Numbers,
Cambridge University Press, 2015. 


\bibitem{St96} J. R. Stembridge,  On the fully commutative elements of Coxeter groups,  J. Algebraic Combin. 5 (4) (1996),  353--385.


\bibitem{St} J. R. Stembridge,  Some combinatorial aspects of reduced words in finite Coxeter groups,   Trans. Amer. Math. Soc.  349 (4)  (1997), 1285--1332. 


\end{thebibliography}
\end{document}